\documentclass[reqno,11pt,a4paper]{amsart}

\usepackage{amsfonts,amsmath,amssymb}
\usepackage[latin1]{inputenc}
\usepackage{dsfont}
\usepackage{enumerate}
\usepackage{xcolor}
\usepackage[all]{xy}
\usepackage{hyphenat}
\usepackage{graphicx}
\usepackage{tikz}
\usepackage{tikz-3dplot}
\usepackage{pgfplots}
\usepackage{multirow}
\usepackage{array}
\usepackage{amssymb}
\usepackage{amsfonts}
\usepackage{mathrsfs}
\usepackage{hyperref}
\usepackage{enumitem}
\usepackage[normalem]{ulem} 
\usepackage[mathscr]{euscript}

\def\br#1\er{\textcolor{red}{#1}}
\def\bb#1\eb{\textcolor{blue}{#1}}

\newcommand{\TT}{\mathrm{T}}
\newcommand{\HH}{\mathrm{H}}
\newcommand{\VV}{\mathrm{V}}
\newcommand{\y}{y}
\newcommand{\ri}{\text{Ric}}
\newcommand{\f}{\rho}
\newcommand{\la}{P}
\newcommand{\hh}{\mid}
\newcommand{\vv}{\cdot}
\newcommand{\di}{\text{div}}
\newcommand{\PP}{\mathfrak{P}}
\newcommand{\vd}{\dot{\partial}}
\newcommand{\vol}{\varSigma^+}
\newcommand{\can}{\mathbb{C}}
\newcommand{\dd}{\mathrm{d}}
\newcommand{\action}{\mathscr{S}}
\newcommand{\ca}{C}

\theoremstyle{definition}
\newtheorem{thm}{Theorem}[section]

\newtheorem{lemma}[thm]{Lemma}
\newtheorem{cor}[thm]{Corollary}
\newtheorem{defi}[thm]{Definition}

\newtheorem{rem}[thm]{Remark}

\title[Schur theorem for weakly Landsberg Finsler metrics]{Schur theorem for the Ricci curvature \\ of any weakly Landsberg Finsler metric}

\author[F. F. Villase\~nor]{Fidel F. Villase\~nor}
\address{Departamento de Geometr\'{\i}a y Topolog\'{\i}a, Facultad de Ciencias \&
\hfill\break\indent IMAG (Centro de Excelencia María de Maeztu)
\hfill\break\indent Universidad de Granada, 18071 Granada, España}
\email{fidelfv@ugr.es}

\begin{document}

\begin{abstract} 
The Ricci version of the Schur theorem is shown to hold for a wide class of Finsler metrics. What is more, let $F$ be any (positive definite) Finsler metric such that $\text{Ric}=\rho F^2$ with $\rho\colon M^n\rightarrow\mathbb{R}$ (i.e., $(M^n,F)$ is Einstein) and $n\geq 3$. For $x\in M$, we express $\text{d}\rho_x$ as an average over the indicatrix in $\text{T}_xM$ of the Hilbert $1$-form weighted by a combination of derivatives of the mean Landsberg tensor. As a consequence of this general expression, if the metric is weakly Landsberg, then $\rho$ must be constant. The proof is based on the invariance of natural functionals under $\text{Diff}(M)$.

Furthermore, we revisit an independent argument which proves the Schur theorem for the class of pseudo-Finsler metrics with quadratic Ricci scalar, improving previous results on the topic. 
\end{abstract}


\maketitle

\tableofcontents

\section{Introduction} 

In 1886, Friedrich Schur showed that if a connected Riemannian manifold $M^n$ with $n\geq 3$ has sectional curvature that varies only as a function of the manifold's points, then it must be constant \cite{Schur}. Later, it was realized that this could be strengthened by using the second Bianchi identity.\footnote{Even though Bianchi published his identities in 1902 \cite{Bianchi}, the contracted ones had already appeared in 1880 by the hand of Aurel Voss \cite{Voss}, and these are enough to derive the result, as we comment on below.} Specifically, if a manifold's Ricci curvature is a function only of its points, then it is constant. This refined result, by closeness to the original one, is referred to as \emph{Schur's theorem} also. 

Being valid in any metric signature (see e.g. \cite{ONeill}), it constraints the possible pseudo-Riemannian geometries of dimension greater than $2$: they cannot have varying Ricci curvature while
maintaining isotropy. From a different viewpoint, it states that the only metrics with isotropic Ricci curvature are the solutions to Einstein's vacuum equation with a cosmological constant. Given the theorem's importance, the problem of its extension to Finsler or Lorentz-Finsler\footnote{The reader is referred to \cite{GGP,Kost,JS,Pfeifer,BJS} for a sample of the growing interest in generalizing the relativistic spacetimes by means of a Finslerian geometry, from the viewpoints of both physics and mathematics.} geometries is a well-known open question \cite{BaoRobles}.

(Pseudo-)Finsler geometry generalizes the
(pseudo-)Riemannian one by allowing for arbitrary (pseudo-)norms instead of just scalar products. In it, there is a notion of curvature tensor \cite[(6.29)]{BejFar} which depends only on the geodesic spray of the pseudo-Finsler function. The flag curvature and Ricci curvature can both be derived
from this curvature tensor. Thus, there are two Finslerian Schur theorems that one might conceive of: the flag curvature version and the Ricci curvature version. The former has been proven for any pseudo-Finsler metric \cite{Berwald,delRiego,Matsumoto,Szilasi}, while the latter is the open problem we are interested in, mainly in the standard Finsler case.

A milestone was Robles', who proved the Ricci version for Randers metrics \cite{Robles}, while many researchers showed it for metrics of other specific forms \cite{YuYou,CST,SY,ZhangXia}. In \cite[Th. 1.1]{SadeghzadehRazaviRezaei}, it was proven for a Finsler manifold which is Landsberg, compact and of so-called \emph{SCR type}. In \cite[Th. 1]{DengKerteszYan}, it was proven for any Berwald Finsler manifold. For completeness, let us mention that Schur-type results have been established also for the mean Berwald curvature \cite{TN,Li}. 

In this article, we present a proof of the Ricci curvature version of the Schur theorem for weakly Landsberg Finsler manifolds, Cor. \ref{schur}, which improves upon \cite[Th. 1.1]{SadeghzadehRazaviRezaei} and \cite[Th. 1]{DengKerteszYan}. We use our main theorem, Th. \ref{main}, to prove it, so our findings can be synthesized as follows.

\begin{quote}
	{\bf Theorem}. On a connected Einstein Finsler manifold $M^n$ of Ricci curvature $\rho\colon M^n\rightarrow\mathbb{R}$, there holds a pointwise relation
	\begin{equation}
		\left(n-2\right)\dd\f_x=\frac{\int_{S_x}\,\theta_{(x,y)}\,d\varSigma_x(y)}{\int_{S_x}\,d\varSigma_x(y)},
		\label{integral identity}
	\end{equation}  
	where $d\varSigma_x$ is the Sasaki-induced volume form on the indicatrix $S_x$ and the $1$-form $\theta$ is a sum of linear and quadratic terms on the mean Landsberg tensor $P_i\,\dd x^i$ and its derivatives. In particular, if $P_i=0$ and $n\geq 3$, then $\f$ is constant.
\end{quote}

Here we explain the underlying philosophy. If one wishes to prove the Ricci-Schur theorem, a first approach would be to use the Finslerian second Bianchi identity \cite[(3.5.3)]{BaoChernShen}, but this is not fruitful \cite[App. B 2]{Robles}. Notwithstanding, in the classical proof, only the contracted identity, or the fact that the Einstein tensor has vanishing divergence, is used. Noticeably, this fact is due to the diffeomorphism invariance of the Einstein-Hilbert action, as is often discussed in the foundations of the general theory of relativity \cite[Ch. 3.3.3]{Str}, \cite[Ch. 9.2]{Bleecker}.
Hence, the classical Schur theorem can be viewed as a consequence of the diff-invariance of a functional, suggesting a possible strategy for generalizations. Recent research on Finslerian gravity has extended the Hilbert functional to pseudo-Finsler geometry\footnote{Actually, this was done for the first time in \cite{ChSh08}, even if implicitly and in the positive definite case.} \cite{PW,HPV19,JSV}, and some of the consequences of the extended invariance have already been explored \cite{HPV22}. It is by continuing this path that we reach the above theorem. Let us emphasize a difference with the classical result. We work on the indicatrix (equiv., projectivized tangent) bundle but we are still forced to use only diffeomorphisms of $M$, so there integral identities such as \eqref{integral identity} arise. Therefore, our method is only suited for metrics that are positive definite on all of $\TT M\setminus\mathbf{0}$, so that the indicatrices are compact. 

Our study is supplemented with a revision of those partial Ricci-Schur theorems in the literature which do hold for pseudo-Finsler metrics \cite[Th. 1]{DengKerteszYan}, \cite[Lem. 3.2]{SadeghzadehRazaviRezaei}, \cite[Th. 3.1]{BacsoRezaei}. We clarify how they can be unified into a single statement, Th. \ref{berwald}, whose only hypothesis is that of the $2$-homogeneous Ricci scalar being quadratic. 

The manuscript is structured as follows. In \S \ref{prelims}, we discuss in detail all the required Finslerian notions, particularly the volume form on the positive projectivization of $\TT M$. In \S \ref{diff invariance}, we define two Finslerian functionals and establish their $\text{Diff}(M)$-invariance and the identities that follow from it. In \S \ref{main results}, we employ them to prove Th. \ref{main} and Cor. \ref{schur}; additionally we prove the mentioned pseudo-Finsler result, Th. \ref{berwald}. Finally, in \S \ref{conclusions}, we provide a summary of all our findings and analyze their implications for the question of a general Finslerian
Ricci-Schur theorem.

\section{Notation and conventions} \label{prelims}

$M$ will always be a smooth manifold of dimension $n\geq 2$ and $a$, $b$, $i$, $j$, $k$, $l$ will be indices in $\left\{1,\ldots,n\right\}$ for which the Einstein summation convention holds. Let $\pi\colon \TT M\rightarrow M$ denote the projection of the tangent bundle, whose elements we regard as pairs $(x,\y)$ where $x\in M$ and $\y\in \TT_xM$. We will also consider a subset $A\subseteq \TT M\setminus\mathbf{0}$ which is open, \emph{conic} (i.e., if $(x,\y)\in A$, then $(x,\mu\y)\in A$ for all $\mu\in\left]0,+\infty\right[$) and with $\pi(A)=M$. In the same way as in \cite[Ch. 1.1]{BaoChernShen}, we denote by $(U,\left(x^i\right))$ an arbitrary coordinate chart for $M$, inducing a natural chart $(\TT U,\left(x^i,\y^i\right))$ for $\TT M$ and, then, for $A$. Whenever convenient, we abbreviate 
\[
\partial_i:=\frac{\partial}{\partial x^i},\qquad\dot{\partial}_i:=\frac{\partial}{\partial\y^i}.
\]

We shall adopt the formalism of \emph{($A$-)anisotropic tensors}\footnote{Which will always be taken to be smooth, i.e. $\mathcal{C}^\infty$ or just as differentiable as necessary, the same as $M$ and $A$.} \cite{Jav19}. These contain the same information as the d-tensors \cite[Ch. 2.5]{BucataruMiron}, but without the need to keep track of wether they live on the horizontal or vertical distributions of $A$, see \cite[\S II C]{HPV22}. Anyway, typically we will work only with their components in natural coordinates, e.g., an anisotropic $1$-form $\theta$ will be characterized by the local functions $\theta_i$ such that $\theta_{(x,\y)}=\theta_i(x,\y)\,\dd x^i$ for $(x,\y)\in A$. The prime example is the \emph{Liouville}, or \emph{canonical}, \emph{anisotropic vector field} $\can$ (cf. \cite{HPV19,HPV22,JSV}):
\[
\can=\can^i\,\partial_i,\qquad\can^i(x,\y):=\y^i\text{ for all }(x,\y)\in \TT M.
\]
As a particular case, the anisotropic functions are just the smooth functions on $A$; we write $\mathcal{F}(A)$ for the set of all of these, which in a natural manner contains $\mathcal{F}(M)$, that of of smooth functions on $M$.\footnote{In fact, as is standard, for any manifold $N$ we will understand that $\mathcal{F}(N)$ is the $\mathbb{R}$-algebra of its smooth functions, $\mathfrak{X}(N)$ the Lie algebra of its tangent vector fields, and $\varOmega(N)$ the graded algebra of its differential forms.} A function $f\in\mathcal{F}(A)$ is \emph{(positively)} $r$-homogeneous if $f(x,\mu\y)=\mu^r f(x,\y)$ for all $(x,\y)\in A$ and $\mu\in\left]0,+\infty\right[$; analogous definition for an anisotropic tensor of arbitrary covariance and contravariance degrees in terms of its components. On the other hand, the \emph{vertical derivative} of an anisotropic tensor is another one of one covariant degree more and whose components we denote by 
\begin{equation}
	f_{\vv j}:=\dot{\partial}_jf=\frac{\partial f}{\partial\y^j},\qquad\theta_{i\,\vv j}:=\dot{\partial}_j\theta_i=\frac{\partial\theta_i}{\partial\y^j},\qquad\ldots
	\label{vertical derivatives}
\end{equation}
\emph{Euler's theorem} lays at the core of the Finslerian theory, relating homogeneity with vertical derivatives and the Liouville field (cf. \cite{BejFar,BaoChernShen,BucataruMiron,JSV,HPV22}):

\begin{thm} \label{euler}
	For $r\in\mathbb{R}$, a function $f\in\mathcal{F}(A)$ is $r$-homogeneous if and only if the contraction of its vertical derivative with $\mathbb{C}$ equals $rf$, that is,
	\[
	\y^if_{\vv i}=rf.
	\]
	Then the vertical derivative of $f$ is $(r-1)$-homogeneous. The same holds for the $r$-homogeneity of $A$-anisotropic tensors of arbitrary type.
\end{thm}

\subsection{Pseudo-Finsler geometry and Einstein metrics}

We refer to \cite{BejFar,Shen} for a systematic treatment of pseudo-Finsler manifolds. 
Many distinguishig features of the Lorentz-Finsler ones are analyzed in \cite{JS2}. However, our main case of interest is the standard Finsler one, covered by \cite{BaoChernShen}. 

\begin{defi} \label{pseudo-finsler}
	Let $A\subseteq \TT M\setminus\mathbf{0}$ be as above. A \emph{pseudo-Finsler metric (defined on $A$)} is a $2$-homogeneous function $L\in\mathcal{F}(A)$ whose \emph{fundamental tensor} $g$, of components
	\[
	g_{ij}(x,\y):=\frac{1}{2}L_{\vv i\vv j}(x,\y)=\frac{1}{2}\frac{\partial^2L}{\partial\y^i\partial\y^j}(x,\y)
	\]
	is non-degenerate at all $(x,\y)\in A$. The pseudo-Finsler metric $L$ is a \emph{(standard) Finsler metric} if it is defined on $\TT M\setminus\mathbf{0}$ and $g$ is positive definite there. Then it follows that $L$ is positive everywhere, so we shall identify it with its ($1$-homogeneous) square root
	\[
	F:=\sqrt{L}\in\mathcal{F}(\TT M\setminus\mathbf{0})
	\]
	and denote by $\text{Fins}(M)$ the set of all of these $F$'s. 
\end{defi}

\begin{rem}
	The main point of this definition is to make apparent that metrics which are positive definite but defined only on a proper $A\subset \TT M\setminus\mathbf{0}$ (such as Kropina's \cite{Kropina}) do not qualify as \emph{Finsler} for us. Most results in \S \ref{diff invariance} and \S \ref{main results} will not apply to them.
\end{rem}

As is typical, we lower and raise indices of anisotropic tensors, resp., with the components $g_{ij}$ and with those of the inverse fundamental tensor, $g^{ij}$, as in the following. The \emph{Cartan tensor} $\ca=\ca_{ijk}\,\dd x^i\otimes\dd x^j\otimes\dd x^k$ of a pseudo-Finsler metric $L$ has as its components
\[
\ca_{ijk}:=\frac{1}{2}g_{ij\,\vv k}=\frac{1}{2}\frac{\partial g_{ij}}{\partial\y^k}.
\]
As these are symmetric, the \emph{mean Cartan tensor} $\text{tr}_g C=C_i\,\dd x^i$ is well-defined, with components
\[
C_i:=C^{a}_{ia}=g^{ab}C_{iab}.
\]
In the Finsler case, we will also need to consider the \emph{Hilbert $1$-form} $\omega$,
\begin{equation}
	\omega_i:=F_{\vv i}=\frac{\y_i}{F}=g_{ia}\frac{\y^a}{F}.
	\label{hilbert form}
\end{equation}

The geodesic equation of the pseudo-Finsler metric $L$ provides its \emph{spray coefficients}, namely
\begin{equation}
	G^i=\frac{1}{4}g^{ic}\left(\frac{\partial g_{cb}}{\partial x^a}+\frac{\partial g_{ac}}{\partial x^b}-\frac{\partial g_{ab}}{\partial x^c}\right)\y^a\y^b.
	\label{spray}
\end{equation}
For the next definitions, we mostly follow \cite{HPV19,HPV22}, particularly \S II B of each of them. We use \eqref{spray} to define the \emph{horizontal subbundle $\HH A\subset \TT A$} associated with the pseudo-Finsler metric $L$. Namely, for $(x,\y)\in A$, we put
\begin{equation}
	\HH _{(x,\y)}A:=\text{Span}\left\{\left.\delta_i\right|_{(x,\y)}\colon i\in\left\{1,\dots,n\right\}\right\},\qquad\delta_i:=\partial_i-G_{\vv i}^a\dot{\partial}_a;
	\label{horizontal}
\end{equation}
this way, a local basis of $\TT A=\HH A\oplus \VV A$ is $\left\{\delta_i,\dot{\partial}_i\colon i\in\left\{1,\dots,n\right\}\right\}$. Moreover, the horizontal subbundle allows one to define Christoffel symbols $\Gamma^i_{jk}$ of the (horizontal part of the) \emph{Chern-Rund connection} (formula below (12) in \cite{HPV19}). We also denote with a vertical bar $_{\hh}$ the index introduced by this covariant derivative to any anisotropic tensor, i.e., by analogy with \eqref{vertical derivatives},
\[
f_{\hh j}:=\delta_jf,\qquad\theta_{i\hh j}:=\delta_j\theta_i-\Gamma_{ji}^k\theta_k,\qquad\ldots
\]
Now, adopting the notation of \cite{CST,SY,Robles} for convenience, a $_{\hh 0}$ will represent the contraction with the Liouville field $\can$ of the new Chern index, i.e., 
\begin{equation}
	f_{\hh 0}:=f_{\hh j}\y^j=\y^j\delta_jf,\qquad\theta_{i\hh 0}:=\theta_{i\hh j}\y^j=\y^j\delta_j\theta_i-\y^j\Gamma_{ji}^k\theta_k,\qquad\ldots
	\label{dynamical}
\end{equation}
(In \cite{BucataruMiron,HPV19,HPV22}, the $_{\hh0}$ is called \emph{dynamical covariant derivative} and denoted by $\nabla$). Finally, we define the \emph{Landsberg tensor} $\la=\la_{ijk}\,\dd x^i\otimes\dd x^j\otimes\dd x^k$ by
\[
\la_{ijk}:=g_{ia}\left(G_{\vv j\vv k}^a-\Gamma_{jk}^a\right),
\]
which are symmetric, and the \emph{mean Landsberg tensor} $\text{tr}_g\la=\la_{i}\,\dd x^i$ by 
\[
\la_i:=\la_{ia}^a=g^{ab}\la_{iab}=G_{\vv a\vv i}^a-\Gamma_{ai}^a.
\]
The pseudo-Finsler metric $L$ is said to be \emph{weakly Landsberg} precisely when its mean Landsberg tensor vanishes, i.e., $\la_i=0$ in any natural coordinates.

\begin{rem} 
	Here we list some elementary properties of the Finslerian objects which we will use without further mention. The Chern derivative is Leibnizian for tensor products and commutes with contractions; consequently, the same is true of $_{\hh 0}$. The $_{\hh i}$ preserves the homogeneity degree of anisotropic tensors, while the $_{\hh 0}$ increases it by one. They satisfy $g_{ij\hh k}=0$, $g^{ij}_{\hh k}=0$ and, then, $g_{ij\hh 0}=0$, $g^{ij}_{\hh 0}=0$, $L_{\hh k}=0$, $L_{\hh 0}=0$. On the other hand, $g^{ij}_{\vv i}=-2\ca^j$ and $\y_iP^i=0$. The canonical field $\can$ is $1$-hom.; $\ca$ and $\text{tr}_g\ca$ are $(-1)$-hom; and $g$, $\omega$, $\la$, and $\text{tr}_g\la$ are $0$-hom.
\end{rem}

The \emph{($2$-homogeneous) Ricci scalar} of $L$ can be defined from its spray coefficients \eqref{spray}:
\begin{equation}
	\ri=2\frac{\partial G^i}{\partial x^i}-\y^j\frac{\partial^2G^i}{\partial x^j\partial\y^i}+2G^j\frac{\partial^2G^i}{\partial \y^j\partial\y^i}-\frac{\partial G^i}{\partial\y^j}\frac{\partial G^j}{\partial\y^i}\in\mathcal{F}(A).
	\label{ricci}
\end{equation}

\begin{rem}
	Note that in this article the Ricci scalar has the same sign as in \cite{Shen} and \cite{JSV}, and it is the opposite of the  scalar of \cite{HPV19}, there denoted $R$. Meanwhile, we have mantained the sign convention of the (mean) Landsberg tensor of \cite{HPV19}, entailing that $P_{ijk}$ (resp. $P_i$) is the opposite of the $\text{Lan}_{ijk}$ (resp. $\text{Lan}_{i}$) of \cite{JSV}. 
\end{rem}

\begin{defi}
	A pseudo-Finsler manifold $(M,L)$ is \emph{Einstein} if its \emph{Ricci curvature}, namely the function defined on $A_{\ast}:=\left\{(x,\y)\in A\colon L(x,\y)\neq 0\right\}$ by $\frac{\ri(x,\y)}{L(x,\y)}$, is actually independent of $\y$. Equivalently,\footnote{This is since each $A_{\ast}\cap \TT_xM$ is dense in $A\cap \TT_xM$ (otherwise the non-degeneracy of $g$ would be contradicted). Anyway, when the metric is Finsler, $A_{\ast}=A=\TT M\setminus\mathbf{0}$ and the Einstein condition is just $\ri=\f F^2$ there.} if there exists some $\f\in\mathcal{F}(M)$ such that $\ri=\f L$.
\end{defi}

\begin{rem}
	Going beyond the Riemannian theory \cite{Besse}, Akbar-Zadeh named \emph{generalized Einstein} those Finsler metrics whose Ricci curvature is isotropic \cite{AkbarZadeh}. Later, it became standard to call them just \emph{Einstein} \cite{Robles,BaoRobles,BacsoRezaei,CST,DengKerteszYan,SadeghzadehRazaviRezaei,SY,YuYou,ZhangXia}. Here, merely for convenience in the exposition, we employ this terminology also in the pseudo-Finsler case. In view of the different studies on Finslerian gravity, \cite{PW,HPV19,Pfeifer,JSV,HPV22}, one should keep in mind that this is in principle independent of the metric solving any proposed Finslerian Einstein equation.
\end{rem}

\subsection{Volume forms, divergences and fiber integrals}

The \emph{(positively) projectivized tangent bundle}\footnote{Consequently with our terminology on homogeneity, we tipically omit the word \emph{positively} since we will not work with the \emph{absolutely projectivized bundle}.} is one of the spaces which clasically have been used as base manifold for the Finslerian geometric objects (see \cite[Ch. 8]{CCL} or \cite[Ch. 2]{BaoChernShen}, where it is called the \emph{projective sphere bundle}). In \cite{HPV19}, it was recognized as a convenient setting for variational Finslerian gravity theories, which inspire our developments, and a detailed construction was given. Here we shall employ the notation of \cite{JSV}, so the canonical projection to the projectivized bundle is
\[
\qquad\;\;\begin{split}
	\mathbb{P}^+\colon \TT M\setminus\mathbf{0}&\longrightarrow\mathbb{P}^+\TT M, \\
	 (x,\y)&\longmapsto\mathbb{P}^+(x,\y)=(x,\mathbb{P}^+\y):=(x,\left\{\mu\y\colon \mu\in\left]0,+\infty\right[\right\}).
\end{split}
\]
Moreover, in this article we abbreviate
\[
\left(\TT M\right)^+:=\mathbb{P}^+\TT M
\]
and its fiber at $x\in M$ by 
\[
\left(\TT_xM\right)^+:=\mathbb{P}^+\TT_xM\subset\left(\TT M\right)^+.
\]
Now we follow the studies in \cite{HPV22,JSV} of how differential forms and divergences on $\left(\TT M\right)^+$ are induced by homogeneous forms and divergences on $\TT M\setminus\mathbf{0}$. Given $F\in\text{Fins}(M^n)$, the \emph{$\left(\TT M\right)^+$-volume form associated with $F$} can be defined as the projectivized volume form provided by the Sasaki metric. More precisely, it is the unique $(2n-1)$-form $\Xi$ such that
\begin{equation}
	\left(\mathbb{P}^+\right)^\ast \Xi=\frac{\det\left(g_{ab}\right)}{F^n}\,\dd^{(n)} x\wedge \imath_{\can^\VV}(\dd^{(n)}\y)
	\label{vol}
\end{equation}
 (as elements of $\varOmega(\TT M\setminus\mathbf{0})$; the $^\ast$ represents a pullback and
\[
\dd^{(n)} x:=\dd x^1\wedge\ldots\wedge\dd x^n,\qquad\imath_{\can^\VV}(\dd^{(n)}\y)=\sum_{i=1}^{n}\left(-1\right)^{i-1}y^i\bigwedge_{j\neq i}\dd \y^j).
\]
 We denote\footnote{Here we find the notation $d\vol$ more convenient than the $dV_0^+$ and $\underline{d\mu}$ appearing in \cite{HPV19,JSV} resp.} 
\[
\Xi=:d\vol\in\varOmega_{2n-1}(\left(\TT M\right)^+).
\]
With respect to it, one defines the \emph{divergence} of any $\mathscr{X}=X^i\delta_i+Y^i\dot{\partial}_i\in\mathfrak{X}(\TT M\setminus\mathbf{0})$ such that the $X^i$ are $0$-hom. and the $Y^i$ are $1$-hom.:
\[
\di(\mathscr{X})\in\mathcal{F}(\left(\TT M\right)^+),\qquad\di(\mathscr{X})\,d\vol:=-\mathfrak{L}_{\mathscr{X}}(d\vol). 
\]
(where $\mathfrak{L}_{\mathscr{X}}$ is the Lie derivative along $\mathscr{X}$ regarded as a vector field on $\left(\TT M\right)^+$, see \cite[\S III]{HPV22} and \cite[Prop. 3]{JSV}). The following divergence formulas have appeared repeatedly in the literature \cite{ChSh08,PW,HPV19,HPV22,JSV}:
\begin{equation}
	\di(X^i\delta_i)=X^i_{\hh i}-P_iX^i,\qquad\di(u\y^i\delta_i)=u_{\hh 0},
	\label{hor divergence}
\end{equation}
\begin{equation}
	\di(Y^i\vd_i)=Y^i_{\vv i}+2C_iY^i-n\frac{\y_i}{F^2}Y^i,
	\label{ver divergence}
\end{equation}
where the homogeneity degrees of the components $X^i$ and $Y^i$ are $0$ and $1$ resp. and that of the function $u$ is $-1$.

Our last aim for this section is to be able to fiberwise integrate and average anisotropic tensor fields on $\left(\TT M\right)^+$. For this, given $F\in\text{Fins}(M^n)$ and a chart $(U,\left(x^i\right))$ for $M$, one locally splits $d\vol$: there exists a unique $\left(n-1\right)$-form $\Theta$ on $\left(\TT U\right)^+$ such that\footnote{This follows from \cite[Prop. 15]{HPV22}, as did the existence and uniqueness of $d\vol$.}
\[
\left(\mathbb{P}^+\right)^\ast \Theta=\frac{\det\left(g_{ij}\right)}{F^n}\imath_{\can^\VV}(\dd^{(n)}\y),\qquad\left.d\vol\right|_{\left(\TT U\right)^+}=\dd^{(n)} x\wedge\Theta.
\]
It will be convenient for us to abuse notation and put
\[
\Theta=:d\vol_x\in\varOmega_{n-1}(\left(\TT U\right)^+).
\]
Now let $f\in\mathcal{F}(\TT U\setminus\mathbf{0})$ be a $0$-homogeneous local function; equiv., $f\in\mathcal{F}(\left(\TT U\right)^+)$. For each fixed $x\in U$, the restriction $\left.f\right|_{\left(\TT_xM\right)^+}$ can be integrated against (the pullback to $\left(\TT_xM\right)^+$ of) $d\vol_x$; such \emph{fiber integral} will be represented by
\[
\int_{\left(\TT_xM\right)^+}f(x,\y)\,d\vol_x(\y).
\]
This way, the use of Fubini's theorem in coordinates allows one to write
\begin{equation}
	\int_{\left(\TT U\right)^+}f\,d\vol=\int_U\left(\int_{\left(\TT_xM\right)^+}f(x,\y)\,d\vol_x(\y)\right)\dd^{(n)}x.
	\label{fiberwise integration}
\end{equation}
However, when the chart is changed, $(U,\left(x^i\right))\rightsquigarrow(\widetilde{U},\left(\widetilde{x}^i\right))$, the fiberwise integral does not provide a well-defined function on $U\cap\widetilde{U}$, but rather a density. Indeed, the transformation law for the fiber volume form is
\[
d\vol_x\rightsquigarrow\widetilde{d\vol_x}=\det\left(\frac{\partial x^i}{\partial \widetilde{x}^j}\right)d\vol_x,
\]
and hence
\[
\int_{\left(\TT_xM\right)^+}f(x,\y)\,\widetilde{d\vol_x}(\y)=\det\left(\frac{\partial x^i}{\partial \widetilde{x}^j}(x)\right)\int_{\left(\TT_xM\right)^+}f(x,\y)\,d\vol_x(\y).
\]
This is fixed by instead considering the \emph{fiberwise average of $f$}, namely
\[
\langle f\rangle\in\mathcal{F}(U),\qquad\langle f\rangle(x):=\frac{\int_{\left(\TT_xM\right)^+}\,f(x,\y)\,d\vol_x(\y)}{\int_{\left(\TT_xM\right)^+}\,d\vol_x(\y)},
\]
so that if $f\in\mathcal{F}(\left(\TT M\right)^+)$, then clearly $\langle f\rangle\in\mathcal{F}(M)$. In particular, all this applies to the local components of any $0$-hom. anisotropic tensor defined on $\TT M\setminus\mathbf{0}$, whose fiberwise averages will turn out to transform as a tensor on $M$. For example, if $\theta=\theta_i\,\dd x^i$ is a $0$-homogeneous $1$-form defined on $\TT M\setminus\mathbf{0}$, then the \emph{fiberwise average of $\theta$} is $\langle\theta\rangle\in\varOmega_1(M)$ defined by
\[
\langle\theta\rangle=\langle\theta\rangle_i\,\dd x^i,\qquad \langle\theta\rangle_i(x):=\langle\theta_i\rangle(x)=\frac{\int_{\left(\TT_xM\right)^+}\,\theta_i (x,\y)\,d\vol_x(\y)}{\int_{\left(\TT_xM\right)^+}\,d\vol_x(\y)}.
\]

\begin{rem}
	Keep in mind that each $F\in\text{Fins}(M)$ provides a natural isomorphism of its indicatrix bundle $\left\{(x,\y)\in \TT M\setminus\mathbf{0}\colon F(x,\y)=1\right\}$ with $\left(\TT M\right)^+$ and this is compatible with (fiber) integration \cite[\S III B 2]{HPV22}. Thus, even though we work on $\left(\TT M\right)^+$ for convenience, it is sensible to regard $\langle\theta\rangle$ as the \emph{average of $\theta$ over indicatrices of $F$}. It also becomes natural to extend the notion of \emph{fiberwise average} to the case in which $\theta$ is homogeneous of an arbitrary degree $r\in\mathbb{R}$, by defining it just as $\langle F^{-r}\theta\rangle$. This way, in the end our results can be thought of in the form of \eqref{integral identity}.
\end{rem}

\begin{rem}
	Despite the notation, our averaging procedure is distinct from others in the literature \cite{GT,MT,Szabo}.
\end{rem}

\section{$\mathrm{Diff}(M)$-invariance of Finslerian functionals} \label{diff invariance}

 As our main result will be a consequence of the diffeomorphism invariance of certain functionals, we find illustrative to start with a brief account of the analogous classical development for the Einstein-Hilbert action. For this, we assume that $M$ is oriented\footnote{Only momentarily and just for simplicity. The manifolds which we will work with later, $\TT M$ and $\left(\TT M\right)^+$, are always orientable.} and follow \cite[\S 4.2]{Berts}, though with our own notation. Let $\lambda$ be the \emph{pseudo-Riemannian Hilbert Lagrangian}, defined by $\lambda[g]:=g^{ij}\text{ric}_{ij}\,d\text{Vol}$. Let $U\subseteq M$ be a precompact open subset and denote $\mathscr{F}^U[g]:=\int_U\lambda[g]$. For any variation of $g$ with variational field $h=h_{ij}\,\dd x^i\otimes\dd x^j$, it is well known that 
\[
\left.\frac{\dd}{\dd t}\mathscr{F}^U[g_t]\right|_{t=0}=-\int_U\left(\text{ric}^{ij}-\frac{1}{2}g^{ab}\text{ric}_{ab}\,g^{ij}\right)h_{ij}\,d\text{Vol}.
\]
 At no point will we require $g$ to be critical for $\mathscr{F}^U$, but, regardless, there is a subset of variations for which this derivative is always $0$. Indeed, for an orientation-preserving $\varphi\in\text{Diff}(M)$, the transformation laws associated with the pullback metric $\varphi^\ast g$ imply that $\lambda$ is diffeomorphism \emph{equivariant}: $\lambda[\varphi^\ast g]=\varphi^\ast\lambda[g]$. In order to see that $\mathscr{F}^U$ is diffeomorphism \emph{invariant}, one restricts to those $\varphi$'s whose support,
 \[
 \text{Supp}\,\varphi:=\overline{\left\{x\in M\colon\,\varphi(x)\neq x\right\}},
 \]
 is contained in $U$, for then $\varphi\colon U\rightarrow U$ and the change of variables theorem gives $\mathscr{F}^U[\varphi^\ast g]=\mathscr{F}^U[g]$. Now one only needs to take a $\xi\in\mathfrak{X}(M)$ with $\text{Supp}\,\xi\subset U$ and consider the variation given by its flow $\varphi_t$, for which $g_t:=\varphi_t^\ast g$ and $h_{ij}=\nabla_j\xi_i+\nabla_i\xi_j$, obtaining with an integration by parts
 \begin{equation}
 	\begin{split}
 		0=\left.\frac{\dd}{\dd t}\mathscr{F}^U[\varphi_t^\ast g]\right|_{t=0}&=-2\int_U\left(\text{ric}^{ij}-\frac{1}{2}g^{ab}\text{ric}_{ab}\,g^{ij}\right)\nabla_j\xi_i\,d\text{Vol} \\
 		&=2\int_U\nabla_j\left(\text{ric}^{ij}-\frac{1}{2}g^{ab}\text{ric}_{ab}\,g^{ij}\right)\xi_i\,d\text{Vol}.
 	\end{split}
 \label{classical identity 0}
 \end{equation}
  The arbitrariness of $\xi$ and $U$ allows to conclude that, whatever the pseudo-Riemannian metric $g$ may be,
  \begin{equation}
  	\nabla_j\left(\text{ric}^{ji}-\frac{1}{2}g^{ab}\text{ric}_{ab}\,g^{ji}\right)=0.
  	\label{classical identity}
  \end{equation}    

 \begin{rem}
 	 This derivation and its Finslerian counterpart, which we develop in this section, can also be viewed as an instance of Noether's second theorem \cite{Noether}. In relation to the $\text{Diff}(M)$ symmetry enabling it, a number of nomenclatures appear in the literature, such as \emph{naturalness} (\cite[App. 3 a]{HPV22}, \cite{Krupka},\cite{KMS}), \emph{general invariance} in the sense of \cite{KT}, or some of the notions of \emph{general covariance} \cite{CG,Anderson,Norton} in physics. Furthermore, distinctions are made between the (ultimately equivalent) \emph{passive} and \emph{active} covariances, see \cite[Ch. 2.5]{Fatibene}. In our treatment, for simplicity, we choose to stick to the terms \emph{diff-equivariance/invariance} and to an active viewpoint. 
 \end{rem} 

We shall prove two lemmas (Lem. \ref{first} and Lem. \ref{second}). It must be stressed that the first of them is already implicit in \cite{HPV19,HPV22}: it could be obtained by formally applying \cite[Th. 31 1.]{HPV22} to the natural Lagrangian \eqref{fhlagrangian}. Nevertheless, we have opted for a self-contained development with the intention of focusing on how the statement and proof must be adapted in the standard Finsler case (Rem. \ref{integration domains}) and arbitrary dimension (Rem. \ref{dim n}). 

 \subsection{First  lemma} 

Some preliminaries are in order. We define the \emph{Finslerian Hilbert Lagrangian} as
\begin{equation}
	\Lambda\colon\text{Fins}(M)\longrightarrow\varOmega_{2n-1}(\left(\TT M\right)^+),\qquad\Lambda[F]=\frac{\ri}{F^2}\,d\vol,
	\label{fhlagrangian} 
\end{equation}
 and, relative to a precompact open subset $D^+\subseteq\left(\TT M\right)^+$, 
 the \emph{Finslerian Hilbert functional} as 
\begin{equation}
	\action^{D^+}\colon\text{Fins}(M)\longrightarrow\mathbb{R},\qquad\action^{D^+}[F]=\int_{D^+}\Lambda[F].
	\label{action}
\end{equation}
 Its variational calculus has already been studied in\footnote{There, some arguments for considering $\action$ as an appropriate generalization of the pseudo-Riemannian functional $\mathscr{F}$ are outlined. One is that (up to a multiplicative constant) $\action^{D^+}$ is variationally equivalent to $\int_{D^+}g^{ij}\ri_{\vv i\vv j}\,d\vol$, which, for Riemannian metrics and $D^+=\left(\TT U\right)^+$, in fact reduces to $\mathscr{F}^U$ upon fiberwise integration \cite[Prop. 6]{HPV19}, \cite[Props. 3.4 and 3.5]{JSV}.} \cite{ChSh08,PW,HPV19,HPV22,JSV}. Here we focus on those variations induced by the action of $\text{Diff}(M)$, which can be regarded as a subgroup of $\text{Diff}(\left(\TT M\right)^+)$ due to the naturalness of $\left(\TT M\right)^+$ over $M$. To be precise, for $\varphi\in\text{Diff}(M)$, we define, respectively, its \emph{lift to $\TT M$} and its \emph{lift to $\left(\TT M\right)^+$} as the maps $\varPhi\in\text{Diff}(\TT M)$ and $\varPhi^+\in\text{Diff}(\left(\TT M\right)^+)$ given by 
 \[
 \varPhi(x,\y):=(\varphi(x),\dd\varphi_x(\y)),\qquad\varPhi^+(x,\mathbb{P}^+\y):=(\varphi(x),\mathbb{P}^+(\dd\varphi_x(\y))),
 \]
  for $x\in M$ and $\y\in\TT_x M$. This way, one has natural commutative diagrams
	\[
	\xymatrix{  \TT M \ar[r]^{\varPhi} \ar[d]_{\mathbb{P}^+} & \TT M \ar[d]\ar[d]^{\mathbb{P}^+} \\ \left(\TT M\right)^+ \ar[r]^{\varPhi^+} \ar[d]_{\pi^+} & \left(\TT M\right)^+ \ar[d]^{\pi^+} \\
	M\ar[r]^{\varphi} & M}
	\]
  in which the horizontal arrows are all diffeomorphisms. 

 This provides pullbacks of all differential forms on $\TT M$ and $\left(\TT M\right)^+$. In particular, we can understand the pullback by $\varphi$ of any $f\in\mathcal{F}(\TT M\setminus\mathbf{0})$ to be $\varPhi^\ast f=f\circ\varPhi$, and in the event that $f$ is $0$-hom. and identified with an $f\in\mathcal{F}(\left(\TT M\right)^+)$, this coincides with $\left(\varPhi^+\right)^\ast f=f\circ\varPhi^+$. If $f=F\in\text{Fins}(M)$, then $\varPhi^\ast F$ also is in $\text{Fins}(M)$ and $\varphi$ is an \emph{isometry} between the Finsler manifolds $(M,\varPhi^\ast F)$ and $(M,F)$. Naturally, the geometric objects of two the isometric manifolds are related by the lifts $\varPhi$ and $\varPhi^+$ in a consistent manner, implying that the Finslerian Hilbert Lagrangian \eqref{fhlagrangian} is \emph{$\mathrm{Diff}(M)$-equivariant}: 
 \begin{equation}
 	\Lambda[\varPhi^\ast F]=\left(\varPhi^+\right)^\ast\Lambda[F].
 	\label{lagrangian equivariance}
 \end{equation}

On the other hand, given the flow $\varphi_t$ generated by some $\xi\in\mathfrak{X}(M)$, the corresponding variational field is\footnote{We follow \cite[Prop. 7]{HPV19}, adopting its notation for the variational field, and use the facts that $\y^i_{\hh j}=0$ and $F_{\hh j}=0$.} 
\begin{equation}
	v:=\frac{1}{2}\left.\frac{\dd}{\dd t}(\varPhi_t^\ast F^2)\right|_{t=0}=\xi_{i\hh j}\y^i\y^j=\left(y_i\xi^i\right)_{\hh 0},
	\label{var field 0}
\end{equation}
which induces 
\begin{equation}
	\frac{v}{F^2}=\left(\frac{\y_i\xi^i}{F^2}\right)_{\hh 0}\in\mathcal{F}(\left(\TT M\right)^+).
	\label{var field}
\end{equation}
The variational formulas of \cite{ChSh08,HPV19,HPV22} are valid as long as one works on an integration domain $D^+$ containing $\text{Supp}(\frac{v}{F^2})$ in its interior, all while the Lagrangian being smooth on the compact $\overline{D^+}\subset \left(\TT M\right)^+$.

\begin{rem} \label{integration domains}
	The only domains that we shall use are those of the form  $D^+=\left(\TT U\right)^+$ with $U\subseteq M$ precompact, open and with smooth boundary; these are well adapted to the above requirements. Indeed, if $\text{Supp}\,\xi\subset U$, then $\text{Supp}(\frac{v}{F^2})\subset\left(\TT U\right)^+$ and any boundary terms (recall \eqref{hor divergence} and \eqref{ver divergence}) depending on $\xi$ will automatically vanish: 
	\begin{equation}
		\quad\quad\int_{\left(\TT  U\right)^+}\di(Y^a[F,\xi]\vd_a)\,d\vol=0,\quad\int_{\left(\TT  U\right)^+}u[F,\xi]_{\hh 0}\,d\vol=0,\quad \ldots
		\label{vanishing boundary terms}
	\end{equation}
	Note that these integration domains are not available for general pseudo-Finsler metrics, which is the reason for having restricted ourselves to the standard Finsler case.\footnote{The formula \cite[(80)]{HPV19} was derived for $\text{Supp}(\frac{v}{L})$ arbitrarily small, which is incompatible with \eqref{var field}. Similarly, \cite[Th. 31]{HPV22} needs of a hypothesis on the support of the Lagrangian on each Lorentz-Finsler indicatrix which would trivialize the Schur theorem that we seek.}
\end{rem} 

\begin{rem} \label{dim n}
	Let us assume the conditions of Rem. \ref{integration domains} and take into account our conventions for $\ri$, $P$ and $v$, as well as $n=\dim M$. From \cite[(2.19)]{ChSh08}, one straightforwardly checks that the variation formula for \eqref{action} under variational fields \eqref{var field} reads
	\begin{equation}
		\begin{split}
			&\quad\left.\frac{\dd}{\dd t}\action^{\left(\TT U\right)^+}[F_t]\right|_{t=0} \\ &=\int_{\left(\TT U\right)^+}\left\{g^{ab}\ri_{\vv a\vv b}-\left(n+2\right)\frac{\ri}{F^2}+2g^{ab}\left(\la_{a\hh b}-\la_a\la_b+\la_{a\hh 0\,\vv b}\right)\right\}\frac{v}{F^2}\,d\vol.
		\end{split}
		\label{var formula}
	\end{equation}
\end{rem}

The equivariance \eqref{lagrangian equivariance} only holds for diffeomorphisms coming from the base manifold $M$; therefore, \eqref{var formula} is only informative when applied with vector fields of components $\xi^i(x)$. It will lead to an equality of the form  
\[
0=\int_U\left(\int_{\left(\TT_xM\right)^+} \ldots\right)_i\xi^i(x)\,\dd^{(n)}x,
\] 
which, in contrast to \eqref{classical identity 0} and \eqref{classical identity}, will produce an integral identity on each fiber of $\left(\TT M\right)^+$ (cf. \cite[Rem. 32]{HPV22}).

\begin{lemma} \label{first}
	Given $D^+=\left(\TT U\right)^+$ as in Rem. \ref{integration domains}, the Finslerian Hilbert functional \eqref{action} is \emph{$\mathrm{Diff}(M)$-invariant}: for $F\in\text{Fins}(M)$ and $\varphi\in\text{Diff}(M)$ with lift $\varPhi\in\text{Diff}(\TT M)$ and $\text{Supp}\,\varphi\subset U$, one has
	\[
	\action^{\left(\TT U\right)^+}[\varPhi^\ast F]=\action^{\left(\TT U\right)^+}[F].
	\]
	Consequently, every Finsler metric has the pointwise property that
		\begin{equation}
			\begin{split}
				\quad&\quad\int_{\left(\TT _xM\right)^+}\left\{g^{ab}\ri_{\vv a\vv b}-\left(n+2\right)\frac{\ri}{F^2}\right\}_{\hh 0}(x,\y)\,\frac{\y_i}{F(x,\y)^2}\,d\vol_{x}(\y) \\
				\quad&=-2\int_{\left(\TT _xM\right)^+}g^{ab}(x,\y)\left(\la_{a\hh b}-\la_a\la_b+\la_{a\hh 0\,\vv b}\right)_{\hh 0}(x,\y)\,\frac{\y_i}{F(x,\y)^2}\,d\vol_{x}(\y) \\
			\end{split}
		\label{identity 1}
		\end{equation}
		 in coordinates around any $x\in M$.
\end{lemma}

\begin{proof}
	Let us parallel the classical derivation with which we started \S \ref{diff invariance}. The first claim follows from the $\mathrm{Diff}(M)$-equivariance \eqref{lagrangian equivariance} of the Finslerian Hilbert Lagrangian by using the change of variables theorem with the lift of $\varphi$ to $\left(\TT M\right)^+$, which satisfies that $\varPhi^+\colon\left(\TT U\right)^+\rightarrow\left(\TT U\right)^+$ due to the hypothesis on the support.
	
	For the second claim, take any $\xi\in\mathfrak{X}(M)$ with $\text{Supp}\,\xi\subset U$
	and represent by $\varphi_t$ its associated one-parameter group. Then $\varphi_t\in\text{Diff}(M)$ with $\text{Supp}\,\varphi_t\subset U$, so for the corresponding variation $F_t:=\varPhi_t^\ast F$ the left hand side of \eqref{var formula} is $0$. The variational field is \eqref{var field}, so integrating by parts in the resulting formula \eqref{var formula} (taking \eqref{vanishing boundary terms} into account) yields
	\begin{equation}
		\begin{split}
			&\quad\int_{\left(\TT U\right)^+}\left\{g^{ab}\ri_{\vv a\vv b}-\left(n+2\right)\frac{\ri}{L}\right\}_{\hh 0}\frac{\y_i\xi^i}{F^2}\,d\vol \\
			&=-2\int_{\left(\TT U\right)^+}g^{ab}\left(\la_{a\hh b}-\la_a\la_b+\la_{a\hh 0\,\vv b}\right)_{\hh 0}\frac{\y_i\xi^i}{F^2}\,d\vol.
		\end{split}
		\label{identity 0}
	\end{equation}
	The only thing that remains in order to get \eqref{identity 1} is to express \eqref{identity 0} in terms of fiber integrals \eqref{fiberwise integration} (assuming that $U$ is a coordinate domain) and then invoke the arbitrariness of $\xi$ and $U$.
\end{proof}

\subsection{Second lemma}

Now, if we just applied \eqref{identity 1} to the Einstein case, $\ri=\f F^2$ with $\f\in\mathcal{F}(M)$, we would get $\y^a\partial_a\f\,\frac{\y_i}{F^2}$ as the integrand on the left hand side, but we are interested in $\partial_i\rho$ itself. Our second lemma provides an expression of the differential of a function which might be of independent interest, and it will be used to rewrite said left hand side.

\begin{lemma} \label{second}
	The functional $\mathscr{I}\colon\mathcal{F}(M)\times\text{Fins}(M)\rightarrow\mathbb{R}$ defined by $\mathscr{I}[\f,F]:=\int \f\,d\vol$ (the integration domain being as in Rem. \ref{integration domains}) is $\text{Diff}(M)$-invariant: in the conditions of Lem. \ref{first},
	\[
	\mathscr{I}[\varphi^\ast \f,\varPhi^\ast F]=\mathscr{I}[\f,F].
	\]
	Consequently, on any Finsler manifold $(M,F)$ of dimension $n$, the differential of any function $\f\in\mathcal{F}(M)$ satisfies the identity
	\begin{equation}
		\partial_i\f(x)\int_{\left(\TT _xM\right)^+}\,d\vol_x(\y)=n\int_{\left(\TT _xM\right)^+}\y^a\partial_a\f(x)\frac{\y_i}{F(x,\y)^2}\,d\vol_{x}(\y) 
		\label{identity 2}
	\end{equation}
	in coordinates around any $x\in M$.
\end{lemma}

\begin{proof}
	The proof proceeds analogously to that of Lem. \ref{first}, so let $\xi\in\mathfrak{X}(M)$ be the corresponding generator of the one-parameter group $\varphi_t$ with lift $\varPhi_t$ and variational field $v$ given by \eqref{var field 0}. Using \eqref{vol}, one has
	\begin{equation}
		\begin{split}
			0&=\int \left.\partial_t\right|_0\left(\f\circ\varphi_t\right)d\vol \\
			&\quad+\int \f\frac{\left.\partial_t\right|_0\det\left(g_{ab}\circ\varPhi_t\right)}{\det\left(g_{ab}\right)}\,d\vol+\int \f\frac{\left.\partial_t\right|_0\left(F^{-n}\circ\varPhi_t\right)}{F^{-n}}\,d\vol \\
			&=\int\xi(\f)\,d\vol+\int \f g^{ab}v_{\vv a\vv b}\,d\vol-n\int \f\frac{v}{F^2}\,d\vol.
		\end{split}
		\label{identity 3}
	\end{equation}
	By a simple computation (using \eqref{ver divergence} and the $2$-hom. of $v$ in the third line),
	\[
	\begin{split}
		\f g^{ab}v_{\vv a\vv b}&=\left(\f g^{ab}v_{\vv a}\right)_{\vv b}-\f g^{ab}_{\vv b}v_{\vv a} \\
		&=\left(\f g^{ab}v_{\vv a}\right)_{\vv b}+2\f C^av_{\vv a} \\
		&=\di(\f g^{ba}v_{\vv a}\dot{\partial}_b)+n\frac{\y_b}{F^2}\f g^{ba}v_{\vv a}=\di(\f g^{ba}v_{\vv a}\dot{\partial}_b)+2n\f\frac{v}{F^2}.
	\end{split}
	\]
	Substituting this back into \eqref{identity 3} and neglecting all the boundary terms (recall \eqref{hor divergence}, \eqref{dynamical} and \eqref{fiberwise integration}),
	\[
	\begin{split}
		0&=\int\xi(\f)\,d\vol+2n\int \f\frac{v}{F^2}\,d\vol-n\int \f\frac{v}{F^2}\,d\vol \\
		&=\int\xi^i\partial_i \f\,d\vol+n\int \f\frac{\left(\y_i\xi^i\right)_{\hh 0}}{F^2}\,d\vol \\
		&=\int\xi^i\partial_i \f\,d\vol-n\int \f_{\hh 0}\frac{\y_i\xi^i}{F^2}\,d\vol \\
		&=\int\left\{\int\left(\partial_i \f(x)-n\y^a\partial_a\f(x)\frac{\y_i}{F(x,\y)^2}\right)d\vol_x(\y)\right\}\xi^i(x)\,\dd^{(n)}x.
	\end{split}	
	\]
	The arbitrariness of $\xi$ and of the coordinate domain allows one to conclude.
\end{proof}

\section{Main results} \label{main results}

\subsection{Schur theorem for weakly Landsberg Finsler metrics}

Our main theorem is a  representation of the differential of the Ricci curvature for any Einstein Finsler manifold (Th. \ref{main}). Since it will follow by particularizing the identities obtained in \S \ref{diff invariance}, let us compute the $\ri$-terms of \eqref{identity 1} in the Einstein case. By substituting $\ri=\f F^2$ there ($\f\in\mathcal{F}(M)$), one gets
	\[
	\begin{split}
		g^{ab}\ri_{\vv a\vv b}-\left(n+2\right)\frac{\ri}{F^2}=g^{ab}\left(\f F^2\right)_{\vv a\vv b}-\left(n+2\right)\f
		&=g^{ab}\left\{2\f g_{ab}\right\}-\left(n+2\right)\f \\
		&=2n\f-\left(n+2\right)\f \\
		&=\left(n-2\right)\f,
	\end{split}
	\]
	from where \eqref{dynamical} and \eqref{horizontal} yield
	\begin{equation}
		\left\{g^{ab}\ri_{\vv a\vv b}-\left(n+2\right)\frac{\ri}{F^2}\right\}_{\hh 0}=\left(n-2\right)\f_{\hh 0}=\left(n-2\right)\y^a\partial_a\f.
		\label{ric terms}
	\end{equation}

As a second step, let us express the integrand in the right hand side of \eqref{identity 1} in a more concise way. We denote 
\begin{equation}
	\PP:=g^{ij}\left(\la_{i\hh j}-\la_i\la_j+\la_{i\hh 0\,\vv j}\right)\in\mathcal{F}(\TT M\setminus\mathbf{0}),	
	\label{P terms = div}
\end{equation}
so that said integrand is $\PP_{\hh 0}\,\frac{\y_i}{F^2}$.

Finally, recall from \S \ref{prelims} the invariant notation for fiberwise averaging associated with the $\left(\TT M\right)^+$-volume form $d\vol$. Again, consider the splitting induced by a coordinate chart $(U,(x^i))$: omitting natural pullbacks, $d\vol=\dd^{(n)}x\wedge d\vol_x$ with $d\vol_x=\frac{\det\left(g_{ij}\right)}{F^n}\,\imath_{\can^\VV}(\dd^{(n)}\y)$. Let $\theta$ be an $r$-homogeneous anisotropic $1$-form defined on $\TT M\setminus\mathbf{0}$. Then the transformation laws of the components $\theta_i$ and of $d\vol_x$ under changes $\left(x^i\right)\rightsquigarrow\left(\widetilde{x}^i\right)$, imply that the fiberwise average
	\[
	\langle\theta\rangle=\langle\theta\rangle_i\,\dd x^i\in\varOmega(M),
	\]
	given (for any $x\in M$) by
	\begin{equation}
		\langle\theta\rangle_i(x):=\langle\theta_i\rangle(x)=\frac{\int_{\left(\TT _xM\right)^+}\,F(x,\y)^{-r}\theta_i (x,\y)\,d\vol_x(\y)}{\int_{\left(\TT _xM\right)^+}\,d\vol_x(\y)},
		\label{averaged form}
	\end{equation}
	is globally well-defined.

\begin{thm} \label{main}
	 Let $(M,F)$ be a Finsler manifold of dimension $n\geq 2$. Assume that $\ri=\f F^2$ with $\f\in\mathcal{F}(M)$. Then the Hilbert form $\omega$ and the function $\PP$ (defined by \eqref{hilbert form}  and \eqref{P terms = div} resp.) satisfy that
	 \[
	 \left(n-2\right)\dd\f=-2n\langle\PP_{\hh 0}\,\omega\rangle;
	 \]
	 equally, in coordinates around each $x\in M$, 
	\begin{equation}
		\left(n-2\right)\partial_i\f (x)=-2n\frac{ \int_{\left(\TT _xM\right)^+}\,F(x,\y)^{-2}\,\PP_{\hh 0}(x,\y)\,\y_i\,d\vol_x(\y) }{\int_{\left(\TT _xM\right)^+}\,d\vol_x(\y)}.
		\label{representation coords} 
	\end{equation}
\end{thm}

\begin{proof}
	In this situation, the conclusions of Lems. \ref{first} and \ref{second} are valid. Taking \eqref{ric terms} and \eqref{P terms = div} into account, the identity \eqref{identity 1} becomes 
	\[
	\begin{split}
		&\left(n-2\right)\int_{\left(\TT _xM\right)^+}\y^a\partial_a\f(x)\frac{\y_i}{F(x,\y)^2}\,d\vol_{x}(\y) \\
		=&-2\int_{\left(\TT _xM\right)^+}\PP_{\hh 0}(x,\y)\frac{\y_i}{F(x,\y)^2}\,d\vol_{x}(\y).
	\end{split}
	\]
	Putting this together with \eqref{identity 2} directly yields \eqref{representation coords}, whose right hand side is exactly $-2n\langle\PP_{\hh 0}\,\omega\rangle_i(x)$ (see \eqref{averaged form} and recall that the $1$-form of components $\PP_{\hh 0}\,\omega_i=\PP_{\hh 0}\,\frac{\y_i}{F}$ is homogeneous of degree $r=1$).
\end{proof}

Thus, on an Einstein Finsler manifold of dimension $3$ or greater, the differential of the Ricci curvature is proportional to $\langle\PP_{\hh 0}\,\omega\rangle$, whereas in dimension $2$ this averaged invariant vanishes. By computing the explicit expression of $\PP_{\hh 0}$ from \eqref{P terms = div}, one obtains the announced Schur theorem under the hypothesis that a certain combination of derivatives of the $\la_i$'s is $0$. 

\begin{cor} \label{schur}
	Let $(M^n,F)$ be a connected Finsler manifold with vanishing mean Landsberg tensor $P_i\,\dd x^i$ or, with more generality,
	\[
	g^{ij}\left(\la_{i\hh j\hh0}-2\la_i\la_{j\hh 0}-\la_{i\hh 0\,\vv j\hh 0}\right)=0.
	\]
	If $\ri=\f F^2$ with $\f\in\mathcal{F}(M)$ and $n\geq 3$, then $\f$ is constant.
\end{cor}

\subsection{Schur theorem for $\mathrm{Ric}$-quadratic pseudo-Finsler metrics}

The method employed for deriving Cor. \ref{schur} cannot be extended to the Lorentz-Finsler or non-standard Finsler cases, not even if the metric is Landsberg or Berwald, see Rem. \ref{integration domains}. Independently, in \cite{DengKerteszYan,SadeghzadehRazaviRezaei}, the Ricci-Schur theorem was proved for Berwald standard Finsler manifolds. 
The authors of \cite{DengKerteszYan} did so by means of Szabó's theorem \cite[Th. 1]{Szabo}, but this is not actually essential (cf. \cite[p. 318]{SadeghzadehRazaviRezaei}) and the proof works in any signature.\footnote{The importance of this point is accentuated by some Lorentz-Finsler metrics being known to violate Szabo's theorem \cite{FHPV}.} Despite this, the corresponding general statement is not explicitly present in the literature as far as we are aware. We end this article by clarifying the relevant extension of \cite[Lem. 3.2]{SadeghzadehRazaviRezaei} and \cite[Th. 1]{DengKerteszYan}, and observing that it turns out to also extend \cite[Th. 3.1]{BacsoRezaei}.

For $A\subseteq \TT M\setminus\mathbf{0}$ as in \S \ref{prelims}, let $f\in\mathcal{F}(A)$ be $2$-homogeneous. We will consider $f$ to be \emph{quadratic} if it is the restriction to $A$ of $h(\can,\can)$, where $h$ is a $2$-covariant tensor field on $M$.  Then, requiring $h$ to be symmetric, we have that for any $(x,\y)\in A$, 
\[
f(x,\y)=h_{ij}(x)\y^i\y^j,\qquad h_{ij}(x)=\frac{1}{2}\frac{\partial^2f}{\partial \y^i\partial\y^j}(x,\y).
\]
This vertical Hessian being independent of $y$ is the well-known necessary and sufficient condition for $f$ to be quadratic. A careful use of it will be key for Th. \ref{berwald}, in particular when $f=L$ is a pseudo-Finsler metric.

\begin{rem} \label{quadratic}
	 Whenever $L$ is Berwald,\footnote{ We take $L$ being \emph{Berwald} to mean that its spray coefficients are quadratic expressions $2G^i(x,\y)=\gamma_{ab}^i(x)\y^a\y^b$. Notwithstanding, if the fibers $A\cap \TT _xM$ are not connected, Th. \ref{berwald} obviously extends to the case in which merely $\frac{\partial^3G^i}{\partial\y^j\partial\y^k\partial\y^l}=0$ or $\frac{\partial^3\ri}{\partial\y^i\partial\y^j\partial\y^k}=0$.} its Ricci scalar $\ri$ is quadratic, without any need of $L$ being affinely equivalent to any quadratic metric as in the conclusion of Szabó's theorem. One sees this in \cite[p. 318]{SadeghzadehRazaviRezaei} or by directly plugging $G^i(x,\y)=\frac{1}{2}\gamma_{ab}^i(x)\y^a\y^b$ into \eqref{ricci}. Moreover, any pseudo-Finsler metric $L\in\mathcal{F}(A)$ which is \emph{$R$-quadratic} in the sense of \cite{BacsoRezaei} also has a quadratic $\ri$.
\end{rem}

\begin{thm} \label{berwald}
	Let $(M,L)$ be a connected pseudo-Finsler manifold with quadratic Ricci scalar.  If $(M^n,L)$ is Einstein with $n\geq 3$, then it is either (globally) Ricci-flat or (globally) pseudo-Riemannian with constant $\frac{\ri}{L}$.
\end{thm}

\begin{proof}
	As above, write $\ri=\f L$ with $\f\in\mathcal{F}(M)$. We will go over the proof in \cite[\S 3]{DengKerteszYan} with the idea of \cite[p. 318]{SadeghzadehRazaviRezaei}, so let us decompose the open set $\left\{x\in M\colon\f(x)\neq 0\right\}$ into its connected components $U_{\alpha}$ and fix one of the indices $\alpha\in\left\{1,2,\ldots\right\}$.\footnote{Notice that in \cite[\S 3]{DengKerteszYan}, $A$ is not the domain of $L$, but a subset of $M$: one of our $U_{\alpha}$'s.} By taking the vertical Hessian of the Einstein condition, one has, for $(x,\y)\in A\cap \TT U_{\alpha}$,
	\begin{equation}
		\frac{1}{2\f(x)}\ri_{\vv i\vv j}(x,\y)=g_{ij}(x,\y),
		\label{L quadratic}
	\end{equation}
	but under our hypotheses the left hand side of this is actually independent of $\y$. This way, as in \cite{SadeghzadehRazaviRezaei}, our equality \eqref{L quadratic} already implies that $L$ is quadratic on $U_\alpha$: there exists a pseudo-Riemannian metric $h^{\alpha}$ there such that $\left.L\right|_{A\cap \TT U_{\alpha}}=\left(h^{\alpha}\right)_{ij}\y^i\y^j$. Denote by $\text{ric}^{\alpha}$ the Ricci tensor of $h^{\alpha}$, so that $\left.\ri\right|_{A\cap \TT U_{\alpha}}=\left(\text{ric}^{\alpha}\right)_{ij}\y^i\y^j$. All that remains is to note that then for $x\in U_\alpha$,
	\[
	\left(\text{ric}^{\alpha}\right)_{ij}(x)=\frac{1}{2}\ri_{\vv i\vv j}(x,\y)=\f(x)g_{ij}(x,\y)=\f(x)\left(h^{\alpha}\right)_{ij}(x),
	\]
	(by taking any $\y\in A\cap \TT _xM$), so the pseudo-Riemannian Schur theorem \cite[Ex. 21 (a)]{ONeill} proves that $\left.\f\right|_{U_{\alpha}}$ is constant. We have shown that $\f$ takes a countable amount of values on $M=\left\{\f=0\right\}\cup U_1\cup U_2\cup\ldots$, thus reaching the conclusion the same way as in \cite{DengKerteszYan}.
\end{proof}

\section{Conclusions} \label{conclusions}

We have given a general description of the differential of the Ricci curvature for Einstein Finsler manifolds of dimension $n\geq 3$ (Th. \ref{main}). As an application, we have established the Schur theorem for such curvature under hypotheses weaker than the metric being weakly Landsberg (Cor. \ref{schur}), generalizing previous results such as \cite[Th 1.1]{SadeghzadehRazaviRezaei} and \cite[Th. 1]{DengKerteszYan}. Our study complements \cite{Robles}, where it was established under the hypothesis of the metric being Randers. As far as we are aware, this provides the only two classes of Finsler metrics containing the Riemannian ones for which the Ricci-Schur theorem is known to be true.

Our proof is variational, stemming from the  invariance under  $\text{Diff}(M)$ of two different Finslerian functionals (Lems. \ref{first} and \ref{second}). By its nature, it cannot be carried out on pseudo-Finsler manifolds, but for them we have observed that the result can be proven whenever the Ricci scalar is quadratic (Th. \ref{berwald}), therefore obtaining a second extension of \cite[Th. 1]{DengKerteszYan} which, independently, extends also \cite[Lem. 3.2]{SadeghzadehRazaviRezaei} and \cite[Th. 3.1]{BacsoRezaei}. Again to the best of our knowledge, this is the only Ricci-Schur theorem which is known to be true for indefinite metrics of an unespecified form.

\section{Acknowledgements}

The author warmly thanks Profs. Miguel Ángel Javaloyes and Miguel Sánchez for the initial conversations which suggested the studied problem and for their valuable revision of the manuscript. He is also grateful to Prof. Nicoleta Voicu for further comments improving this article, and specially for the discussions during a stay in the Transilvania University in Brasov which led to a much better understanding of the methods employed here.

This work was partially supported by the FPU grant (Formación de Profesorado Universitario) with reference number FPU19/01009 from the Spanish Ministerio de Universidades, by the project PID2020-116126GB-I00 funded by MCIN/AEI/10.13039/501100011033, by the project PY20-01391 (PAIDI 2020) funded by Junta de Andalucía-FEDER, and by the framework of IMAG-María de Maeztu grant CEX2020-001105-M funded by MCIN/AEI/\\10.13039/50110001103.

\end{document}